\documentclass{amsart}
\usepackage{amsmath, amssymb, amsthm,xcolor}
\usepackage{relsize}
\newtheorem{thr}{Theorem}

\newtheorem{cor}[thr]{Corollary}
\newtheorem{stat}[thr]{Proposition}
\newtheorem{obs}[thr]{Observation}
\newtheorem{claim}[thr]{Claim}

\theoremstyle{definition}

\theoremstyle{remark}

\numberwithin{equation}{section}

\begin{document}%\color{blue}

%\title[A tropical approach to supertropical linear algebra]{A tropical approach to \\ supertropical linear algebra}
%\title{A $0.1654n^3$ upper bound for the reset threshold of an $n$-state synchronizing automaton}
\title[An upper bound for synchronizing words of finite automata]{An improvement to a recent upper bound for synchronizing words of finite automata}

%    author one information
\author{Yaroslav Shitov}
%\address{National Research University Higher School of Economics, 20 Myasnitskaya Ulitsa, Moscow 101000, Russia}
\email{yaroslav-shitov@yandex.ru}

%    \subjclass is required.
\subjclass[2010]{68Q70}

\keywords{Automata theory, \v{C}ern\'{y} conjecture}

%\date{October 3, 2011 and, in revised form, February 22, 2012}

%\dedicatory{}

%    "Communicated by" -- provide editor's name; required.
%\commby{Jim Haglund}

%    Abstract is required.
\begin{abstract}
It has been known since the 60's that any complete discrete $n$-state automaton admits a reset word of length not exceeding $\alpha n^3+o(n^3)$ for some absolute constant $\alpha$. J.-E.~Pin and P.~Frankl proved this statement with $\alpha=1/6=0.1666...$ in 1982, and this bound remained best known until 2017, when M. Szyku\l{}a decreased its value to $\alpha\approx0.1664$. In this note, we present a modification to the latest approach and develop a different counting argument which leads to a more substantial improvement of $\alpha\leqslant 0.1654$.
\end{abstract}

\maketitle



\section{Introduction}

Let $\mathcal{A} = (Q, \Sigma, \delta)$ be a \textit{deterministic finite automaton}, where $Q$ is a finite set of \textit{states}, $\Sigma$ is a finite \textit{alphabet}, and $\delta: Q\times\Sigma\to Q$ is a \textit{transition function}, which assigns a mapping $Q\to Q$ to every letter of $\Sigma$. This function naturally extends to an action $Q\times\Sigma^*\to Q$ of the free monoid $\Sigma^*$ on $Q$, and this action is still denoted by $\delta$. For a subset $S\subseteq Q$ and a word $w\in\Sigma^*$, we define $S\cdot w$ as the set of all images $s\cdot w$ of elements $s\in S$ under $w$. The cardinality of $Q\cdot w$ is called the \textit{rank} of a word $w$, and the \textit{rank of an automaton} is defined as the smallest possible rank of a word. An automaton $\mathcal{A}$ of rank one is called \textit{synchronizing}, and the length of the shortest rank-one words is called the \textit{reset threshold} of $\mathcal{A}$ and denoted by $\operatorname{rt}(\mathcal{A})$.

\smallskip

Upper bounds for reset thresholds of synchronizing automata were a topic of extensive research in the last 50 years, and one of the main goals of this study is a famous conjecture stating that $\operatorname{rt}(\mathcal{A})\leqslant (n-1)^2$ for any synchronizing $n$-state automaton $\mathcal{A}$; this statement was considered many years ago by different authors and became known as the \textit{\v{C}ern\'{y} conjecture} (see a historical survey in~\cite{Volkov}). There are a lot of progress on this question for different special classes of automata~\cite{Pin2, Stein, TrahAper}, but the general version of the \v{C}ern\'{y} conjecture remains wide open. The cubic upper bounds on the reset threshold, that is, inequalities of the form $\operatorname{rt}(\mathcal{A})\leqslant \alpha n^3 +o(n^3)$ for some fixed $\alpha$, have been known since 1966, see~\cite{Starke}. After a series of improvements~\cite{Cer, Cer2, Pin1, Pin2}, the progress stuck for 35 years on the celebrated $\alpha=1/6=0.1666...$ bound of J.-E. Pin and P. Frankl~\cite{Frankl, Pin2}. In 2011, A. Trahtman~\cite{Trah} discovered an idea of how to find a relatively short word of rank at most $n/2$, and M. Szyku\l{}a~\cite{Sz} combined it with a neat linear algebraic argument and finally improved the upper bound to $\alpha\approx 0.1664$ in 2017. The purpose of this note is to modify the approach of~\cite{Sz} and get a more substantial improvement of $\alpha\leqslant 0.1654$.

\section{Modifying the method}

From now on, we denote by $\mathcal{A}=(Q,\Sigma,\delta)$ a synchronizing automaton with $n$ states, and we define the \textit{corank} of a word $w\in\Sigma^*$ as $n-\operatorname{rk}w$. Our aim is to give a relevant modification of the following proposition, which plays a crucial role in~\cite{Sz}.

\begin{thr}\label{thrSz1}\normalfont{(Lemma~2 in~\cite{Sz}.)}
\textit{Let $A$ and $S$ be subsets of $Q$ satisfying $\varnothing\subsetneq A\subsetneq S$. Suppose that there is a word $w\in \Sigma^*$ such that $A\nsubseteq S \cdot w$. Then there exists a word $w$ of length at most $n-|A|$ satisfying either (1) $A\nsubseteq S\cdot w$ or (2) $|S\cdot w|<|S|$.}
\end{thr}

%The following statement does not appear in his paper explicitly, but we reformulate it in a way that is more convenient in our argument below.

%\begin{obs}\label{obsFrob}
%Let $u,v,w\in\Sigma^*$ and $\operatorname{rk}(u\cdot v\cdot w)<\operatorname{rk}(u\cdot v)$. Then $\operatorname{rk}(v\cdot w)<\operatorname{rk}(v)$.
%\end{obs}

%\begin{proof}
%Follows from the Frobenius inequality $\operatorname{rk}(AB)-\operatorname{rk}(ABC)\leqslant\operatorname{rk}(B)-\operatorname{rk}(BC)$ valid for any triple $(A,B,C)$ of $n\times n$ square matrices.
%\end{proof}

%As we will see, the construction of short words $w'$ satisfying $A\nsubseteq S\cdot w'$ is important for reducing the
In~\cite{Sz}, a successive application of Theorem~\ref{thrSz1} was used to construct a word $\omega$ that satisfies $A\nsubseteq S\cdot \omega$. At the first glance, the following theorem may look like a mere reformulation of this technique avoiding a direct mention of a successive application. However, it gives a more explicit description of a desired word $\omega$ which will give a significant improvement on the bound of~\cite{Sz} later in this note.

\begin{thr}\label{thrmy2}%\normalfont{(See~\cite{Sz}.)}
Let $u\in\Sigma^*$ be a word of length $l$ and corank $r\in[1, n/2-1]$. Assume that, for an integer $\lambda$, there exists a word $v$ of length $\lambda$ such that $\operatorname{rk}v\leqslant\operatorname{rk}v'$ for any word $v'$ of length at most $\lambda+2r$. Then there is a word of length at most $l+\lambda+2r$ and corank at least $r+1$.
\end{thr}

\begin{proof}
%We are already done if $\operatorname{rk}(u\cdot v)<n-r$, so we can assume without loss of generality that $\operatorname{rk}(u\cdot v)=n-r$. Let $A\subset Q\cdot u\cdot v$ be the set of states whose preimage under the mapping $q\to q\cdot u\cdot v$ is unique. According to Lemma~7 in~\cite{Sz}, one has $|A|\geqslant n-2r$. Using Theorem~\ref{thrSz1}, one finds a word $w$ of length at most $2r$ such that either (1) $A\nsubseteq Q\cdot u\cdot v\cdot w$ or (2) $|Q\cdot u\cdot v\cdot w|<|Q\cdot u\cdot v|$. The possibility (2) implies $\operatorname{rk}(v\cdot w)<\operatorname{rk} (v)$ by Observation~\ref{obsFrob}, and so it does not realize by the assumption of the current theorem. So we have $A\nsubseteq Q\cdot u\cdot v\cdot w$
For a state $\sigma$ in $Q\cdot u$, we denote by $\sigma\cdot u^{-1}$ the preimage of $\sigma$ under the mapping $q\to q\cdot u$. Let $A$ be the union of all those preimages $\sigma\cdot u^{-1}$ which are singleton sets; according to Lemma~7 in~\cite{Sz}, one has $|A|\geqslant n-2r$.

Now we want to find a word $w$ of length at most $2r$ such that $A\nsubseteq Q\cdot v\cdot w$, which would allow us to find an element $a\in A$ satisfying $a\notin Q\cdot v\cdot w$, which would imply $a\cdot u\notin Q\cdot v\cdot w\cdot u$ and thus lead to a desired conclusion $Q\cdot u\supsetneq Q\cdot v\cdot w\cdot u$. Such a word $w$ is found immediately if $A$ and $S:=Q\cdot v$ satisfy the assumptions of Theorem~\ref{thrSz1}, because the second possibility of its conclusion means that $|Q\cdot v\cdot w|<|Q \cdot v|$ and contradicts the assumption of the current theorem.

As to the assumptions of Theorem~\ref{thrSz1}, the one in the second sentence holds because our automaton is synchronizing. In particular, there should be a letter $b\in\Sigma$ such that $A\cdot b\neq A$. So if $A$ was equal to $Q\cdot v$, then we could have taken $w=b$ and proceed as in the previous paragraph, and, similarly, if $A$ was not a subset of $Q\cdot v$ at all, we could have taken $w$ to be the empty word and do the same thing.
%For a state $\sigma$ in $Q$, we denote by $\sigma\cdot u^{-1}$ the preimage of $\sigma$ under the mapping $q\to q\cdot u$. Let $A\subset Q\cdot u$ be the union of all those preimages $\sigma\cdot u^{-1}$ which are singleton sets; according to Lemma~7 in~\cite{Sz}, one has $|A|\geqslant n-2r$. Using Theorem~\ref{thrSz1} with $S=Q\cdot v$, one finds a word $w$ of length at most $2r$ such that either (1) $A\nsubseteq Q\cdot v\cdot w$ or (2) $|Q\cdot v\cdot w|<|Q \cdot v|$. The possibility (2) contradicts the assumption of the current theorem, so there exists a word $a$ which belongs to $A$ but not to $Q \cdot v\cdot w$, which implies $a\cdot u\notin Q\cdot v\cdot w\cdot u$ and thus $Q\cdot u\neq Q\cdot v\cdot w\cdot u$.
\end{proof}

One more theorem is needed before we can proceed to counting --- we cannot improve on the Pin--Frankl bound without using the Pin--Frankl bound.

\begin{thr}\cite{Frankl, Pin2}\label{thrPF}
Let $u\in\Sigma^*$ be a word of length $l$ and corank $r\leqslant n-2$. Then there is a word of length at most $l+(r+1)(r+2)/2$ and corank at least $r+1$.
\end{thr}

\section{Counting}

As Theorem~\ref{thrmy2} suggests, we are going to study the gaps between the smallest lengths of words with ranks taking consecutive pairs of values. Formally speaking, we denote by $\lambda_i$ the smallest length of a word with corank at least $i\in\{0,\ldots,n-1\}$; we obviously have $0=\lambda_0\leqslant\ldots\leqslant\lambda_{n-1}=\operatorname{rt}(\mathcal{A})$. We also write $\lambda_n=+\infty$ and define $\rho$ as the smallest corank satisfying $\lambda_{\rho+1}-\lambda_\rho>n$.

\begin{obs}\label{obsmy5}
We have $\lambda_\rho<n^2$.
\end{obs}

Further, we set $\delta_j=\lambda_{j+1}-\lambda_{j}$ for any $j\in\{0,\ldots,\rho\}$, and, for any integer $r\leqslant n/2$, we define the quantity $s_r$ as the number of those $j\in\{0,\ldots,\rho\}$ which satisfy $\delta_j\in\{2r-1,2r\}$. Let us translate Theorems~\ref{thrmy2} and~\ref{thrPF} to this language.

\begin{thr}\label{thrmy4}
Let $u\in\Sigma^*$ be a word of length $l$ and corank $r\in[1, n/2-1]$. Then there is a word of corank at least $r+1$ and length not exceeding
$$l+\min\left\{\frac{(r+1)(r+2)}{2},\,\,\,2(s_1+2s_2+3s_3+\ldots+rs_r)+2r\right\}.$$
\end{thr}

\begin{proof}
%We are allowed to put first argument of $\min$ by Theorem~\ref{thrPF} immediately. In order to explain the second argument by an application of Theorem~\ref{thrmy2}, we need a word $v$ of length $h\leqslant 2(s_1+2s_2+3s_3+\ldots+rs_r)$ whose rank does not exceed the rank of any word of length at most $h+2r$. According to our notation, we can pick $v$ to be a word of corank $\tau$ and length $\lambda_\tau$, where $\tau$ is the minimal index for which $\delta_\tau$ exceeds $2r$.
We are allowed to put the first argument of $\min$ by Theorem~\ref{thrPF} immediately. Further, let us pick a word $v$ of corank $\tau$ and length $\lambda_\tau$, where $\tau$ is the minimal index for which $\delta_\tau$ exceeds $2r$. The length of $v$ does not exceed the sum of all the $\delta_j$'s not exceeding $2r$, which is at most $2(s_1+2s_2+3s_3+\ldots+rs_r)$. Also, we cannot get a word of rank less than $\operatorname{rk} v$ unless we take $\delta_\tau>2r$ more letters than $v$ has --- therefore, we can apply Theorem~\ref{thrmy2} and justify the second argument of $\min$.
\end{proof}

\begin{cor}\label{cormy6}
The reset threshold of $\mathcal{A}$ does not exceed
$$\frac{7}{48}n^3+2\mathlarger{\sum}\limits_{r=\rho}^{\lfloor n/2\rfloor}\min\left\{\frac{r^2}{4},1s_1+\ldots+rs_r\right\}+3n^2.$$
\end{cor}

\begin{proof}
We use Observation~\ref{obsmy5} to get a word of corank at least $\rho$ and length at most $n^2$, then we upgrade it to a word of rank $\leqslant\lceil n/2\rceil$ by a successive application of Theorem~\ref{thrmy4}, and then we construct a synchronizing word by Theorem~\ref{thrPF} (the cost of this last step is $\leqslant7n^3/48$ additional letters). Also, the expressions under the minimum were simplified by isolating the $O(n^2)$ terms in the last summand.
\end{proof}

The following statement is going to complete the proof of our main result.

\begin{stat}\label{mystat7}
Let $n$ and $\rho<n/2$ be positive integers, let $k=\lfloor n/2\rfloor$. Let $s_1,\ldots,s_k$ be nonnegative real numbers satisfying $s_1+\ldots+s_k\leqslant\rho$. Then
\[\varphi(s_1,\ldots,s_k):=\mathlarger{\sum}\limits_{r=\rho}^{k}\min\left\{\frac{r^2}{4},1s_1+\ldots+rs_r\right\}\leqslant \frac{15\,625\,n^3}{1\,597\,536}+o(n^3).\tag{1} \label{eq1}\]
%\[\varphi(s_1,\ldots,s_k)=\mathlarger{\sum}\limits_{r=\rho}^{\lfloor n/2\rfloor}\min\left\{\frac{r^2}{4},1s_1+\ldots+rs_r\right\}\leqslant \frac{15\,625\,n^3}{1\,597\,536}+o(n^3).\]
\end{stat}

The numbers $s_1,\ldots,s_k$ appearing in Corollary~\ref{cormy6} are clearly nonnegative and have the sum not exceeding $\rho$, so we can apply Proposition~\ref{mystat7} and get
\[\left(\frac{7}{48}+\frac{2\cdot 15\,625}{1\,597\,536}\right)n^3+o(n^3)\tag{B}\label{the bound}\]
or $0.1654n^3+\operatorname{O}(1)$ as an upper bound for the reset threshold of $\mathcal{A}$.

\section{Proving Proposition~\ref{mystat7}}

The last section is devoted to a solution of the optimization problem appearing in Proposition~\ref{mystat7}. First, we restrict our attention to a certain special case.

\begin{claim}\label{cl1}
It is sufficient to prove Proposition~\ref{mystat7} under the additional assumptions of $s_1=\ldots=s_{\rho-1}=0$ and
\[1s_1+\ldots+\tau s_\tau\leqslant \tau^2/4\,\,\,\,\,\,\,\,\mbox{for all $\tau\in\{\rho,\ldots,k\}$} \tag{$2_\tau$} \label{1tau}\]
(where the latter says that the minimum is always attained at the second argument).
\end{claim}

\begin{proof}
Let us define $s_r'=0$ for $r<\rho$, $s_r'=s_r$ for $r>\rho$, and $s_\rho'=(1s_1+\ldots+\rho s_\rho)/\rho$. The new values are nonnegative and sum to at most $s_1+\ldots+s_k\leqslant\rho$, so they satisfy the assumptions of Proposition~\ref{mystat7}. Also, we have $1s_1+\ldots+rs_r=1s_1'+\ldots+rs_r'$ for all $r\geqslant\rho$, so the arguments of the minimuma do not change, and we can pass to $(s_1',\ldots,s_k')$ without loss of generality.

Now, for a tuple $s=(s_1,\ldots,s_k)$ not satisfying one of the conditions~\eqref{1tau}, we define $t(s)$ as the smallest $\tau$ for which it fails. In other words, we have $$\alpha:=1s_1+\ldots+ts_t>t^2/4\,\,\,\,\,\,\mbox{and}\,\,\,\,\,\,1s_1+\ldots+rs_r\leqslant r^2/4\,\,\,\,\,\,\mbox{for all}\,\,\,\,\,\,r<t.$$ Then we set $s_t'=s_t+t/4-\alpha/t$ and $s_r'=s_r$ for $r\notin\{\tau,\tau+1\}$,
and also, if $t\neq k$, we define $s_{t+1}'=s_{t+1}-t/4+\alpha/t$. The values $s_r'$ are again nonnegative and sum to at most $s_1+\ldots+s_k\leqslant\rho$, so they satisfy the assumptions of Proposition~\ref{mystat7}. Also, we have $\varphi(s')\geqslant\varphi(s)$ because the summands with $r\in[\rho,t]$ did not change while the $(t+1)$st and later summands could not have decreased. Finally, we note that, even if $s'$ still does not satisfy some of the conditions~\eqref{1tau}, we still can prove the second statement of the current theorem by induction because $t(s')>t(s)$.\end{proof}

From now on, we assume that the conditions $s_1=\ldots=s_{\rho-1}=0$ and~\eqref{1tau} hold; we also recall that $s_r\geqslant0$ and $s_1+\ldots+s_k\leqslant\rho$. We call the set of all tuples $(s_1,\ldots,s_k)$ satisfying these conditions a \textit{feasible set}; Claim~\ref{cl1} allows us to restrict Proposition~\ref{mystat7} to it. 
%$$\varphi(s_1,\ldots,s_k)=\sum\limits_{r=\rho}^k (k-r+1)r s_r.$$
The feasible set is compact (for any fixed $\rho$), so the function $\varphi$ should have a maximum point $\sigma=(0,\ldots,0,\sigma_\rho,\ldots,\sigma_k)$. Let $\beta$ and $\gamma$ be, respectively, the minimal and maximal indices $i$ satisfying $\sigma_i\neq0$. 

\begin{claim}\label{cl2}
We have $1\sigma_1+\ldots+r\sigma_r=r^2/4$ for all $r\in\{\beta+1,\ldots,\gamma-1\}$.
\end{claim}

\begin{proof}
Assume the converse and find the maximal $\nu\in\{\beta+1,\ldots,\gamma-1\}$ for which $1s_1+\ldots+\nu s_\nu<\nu^2/4$. Then we pick a sufficiently small $\varepsilon>0$ and define $$\sigma'_\beta=\sigma_\beta-\varepsilon,\,\,\,\,\ \sigma'_\nu=\sigma_\nu+\varepsilon(\nu-\beta+1),\,\,\,\,\,\sigma'_{\nu+1}=\sigma_{\nu+1}-\varepsilon(\nu-\beta)$$
and $\sigma'_r=\sigma_r$ for $r\notin\{\beta,\nu,\nu+1\}$. The tuple $\sigma'=(\sigma_1',\ldots,\sigma_k')$ sums to $\sigma_1+\ldots+\sigma_k\leqslant\rho$, and its coordinates are nonnegative for a sufficiently small $\varepsilon$. Further, the sum as in~\eqref{1tau} could not have increased for $\tau<\nu$; the same sum but with $\tau=\nu$ has changed by something proportional to $\varepsilon$ and so it could not have overcome $\nu^2/4$. Finally, such a sum with $\tau>\nu$ did not change as we can check directly, so the tuple $\sigma'$ belongs to the feasible set. Finally, we have \[\varphi(\sigma)=\sum\limits_{r=\rho}^k (k-r+1)\,r\,\sigma_r\]%\tag{6}\label{eqeqeq}\] 
if $\sigma$ is in the feasible set, and one can get the inequality $\varphi(\sigma')>\varphi(\sigma)$, which contradicts the maximality of $\sigma$. (Such an inequality can be deduced by a straightforward computation, but let us point it out that it follows from the strict concavity of the sequence of the coefficients of $\sigma_\rho,\ldots,\sigma_k$ in the above expression for $\varphi$.) %~\eqref{eqeqeq}.
\end{proof}

Now we are going to employ Claim~\ref{cl2} to complete the proof of Proposition~\ref{mystat7}. First, we have $r\sigma_r=r^2/4-(r-1)^2/4$ or \[\sigma_r=0.5-0.25/r\,\,\,\,\mbox{for all}\,\,\,\,r\in\{\beta+2,\ldots,\gamma-1\}.\tag{3} \label{eq2}\]
Secondly, we have $\beta\sigma_\beta+(\beta+1)\sigma_{\beta+1}=(\beta+1)^2/4$ or \[\sigma_\beta+\sigma_{\beta+1}\geqslant0.25(\beta+1).\tag{4} \label{eq3}\] Summing the inequalities~\eqref{eq2} over all $r\in\{\beta+2,\ldots,\gamma-1\}$ and adding the inequality~\eqref{eq3}, we get at most $\sigma_1+\ldots+\sigma_k\leqslant\rho\leqslant\beta$ on the left-hand side, or
%\[\beta\geqslant \frac{\beta+1}{4}+\frac{\gamma-\beta-2}{2}-\frac{\ln n}{4},\tag{4} \label{eq4}\]
\[\beta\geqslant 0.25(\beta+1)+0.5(\gamma-\beta-2)-0.25 \ln n.\tag{5} \label{eq4}\]
Now we estimate $\varphi(\sigma)$ using Claim~\ref{cl2} again. Clearly, the summands corresponding to $r<\beta$ are zero in the definition~\eqref{eq1} of $\varphi$, the summands with $r\in\{\beta,\ldots,\gamma\}$ cannot exceed $0.25r^2$, and the summands with $r>\gamma$ are at most $0.25\gamma^2$. We get $\varphi(\sigma)\leqslant\sum_{r=\beta}^\gamma 0.25r^2+0.25(0.5n-\gamma)\gamma^2$, or
\[\varphi(\sigma)\leqslant\psi(\beta,\gamma):=(-2 \beta^3 - 4 \gamma^3 + 3 \gamma^2 n + 6 \gamma^2 + 6 \gamma + 2)/24.\]%\tag{6} \label{eq5}\]

It remains to maximize $\psi(\beta,\gamma)$ subject to $0\leqslant\beta\leqslant\gamma\leqslant0.5n$ and~\eqref{eq4}; these inequalities define a quadrilateral $\Delta$ with vertices $(0,0)$, $(0,0.5\ln n+1.5)$, $(0.5n,0.5n)$, $(0.2n-0.2\ln n-0.6,0.5n)$. Being strongly monotone in $\beta$, the function $\psi$ cannot have a maximum inside $\Delta$, so it remains to perform a basic calculus task and maximize $\psi$ on the edges. The computation shows that the maximum of $\psi$ on $\Delta$ is attained at the point $(25n/129,125n/258)+o(n)$ and confirms that it is equal to the right-hand side of~\eqref{eq1}. Therefore, the proof of our main result is complete.

\medskip

As a final remark, let us note that we were not interested to optimize the $o(n^3)$ term of our bound~\eqref{the bound}, but a direct computation of the previous paragraph gives an $O(n^2\log n)$ estimate of this term. A more careful application of Claim~\ref{cl2} would make it $O(n^2)$, with explicit and reasonably small coefficients of the powers of $n$.


\begin{thebibliography}{99}

\bibitem{Cer}
J. \v{C}ern\'{y}, Pozn\'{a}mka k homog\'{e}nnym eksperimentom s kone\v{c}n\'{y}mi automatami, \textit{Matematicko-fyzik\'{a}lny \v{C}asopis Slovenskej Akad\'{e}mie Vied} 14.3 (1964) 208--216. (In Slovak.)

\bibitem{Cer2}
J. \v{C}ern\'{y}, A. Pirick\'{a}, B. Rosenauerov\'{a}, On directable automata, \textit{Kybernetica} 7 (1971) 289.

\bibitem{Frankl}
P. Frankl, An extremal problem for two families of sets, \textit{European J. Combin.} 3 (1982) 125.

\bibitem{GJT}
F. Gonze, R. M. Jungers, A. N. Trahtman, A note on a recent attempt to improve the Pin-Frankl bound, \textit{Discrete Mathematics \& Theoretical Computer Science} 17.1 307--308 (2015).

\bibitem{Pin1}
J.-E. Pin, Sur les mots synchronisants dans un automate fini, \textit{Elektron. Informationsverarb. Kybernet.} 14 (1978) 293--303. (In French.)

\bibitem{Pin2}
J.-E. Pin, On two combinatorial problems arising from automata theory, in Proc. 2nd ICGT, \textit{North-Holland Mathematics Studies} 75 (1983) 535--548.

\bibitem{Starke}
P. H. Starke, Eine Bemerkung \"{u}ber homogene Experimente. \textit{Elektron. Informationsverarb. Kybernet.} 2 (1966) 257--259. (In German.)

\bibitem{Stein}
B. Steinberg, The \v{C}ern\'{y} conjecture for one-cluster automata with prime length cycle, \textit{Theor. Comput. Sci.}412.39 (2011) 5487--5491.

\bibitem{Sz}
M. Szyku\l{}a, Improving the Upper Bound on the Length of the Shortest Reset Words, in Proc. 35th STACS, \textit{LIPIcs} 96 (2018) 56.

\bibitem{TrahAper}
A. N. Trahtman, The \v{C}ern\'{y} conjecture for aperiodic automata, \textit{DMTCS} 9.2 (2007) 3--10.

\bibitem{Trah}
A. N. Trahtman, Modifying the upper bound on the length of minimal synchronizing word, in Proc. 19th FCT, \textit{LNCS} 6914 (2011) 173--180.

\bibitem{Volkov}
M. V. Volkov, Synchronizing Automata and the \v{C}ern\'{y} Conjecture, in Proc. 2nd LATA, \textit{LNCS} 5196 (2008) 11--27.


\end{thebibliography}
\end{document}